\documentclass[12pt,a4paper]{article}
\usepackage[T1]{fontenc}
\usepackage[utf8]{inputenc}
\usepackage{lmodern,amsmath,amssymb,amsthm,mathtools}
\usepackage[margin=28mm]{geometry}
\usepackage{microtype}
\usepackage[shortlabels]{enumitem}
\usepackage{xcolor}
\usepackage[colorlinks=true,linkcolor=blue!45!black,citecolor=blue!45!black,urlcolor=blue!45!black]{hyperref}
\hypersetup{pdftitle={Every commutative perfect ring is a quotient of an almost perfect domain},pdfauthor={}}
\numberwithin{equation}{section}
\newtheorem{theorem}{Theorem}[section]
\newtheorem{lemma}[theorem]{Lemma}
\newtheorem{proposition}[theorem]{Proposition}
\newtheorem{corollary}[theorem]{Corollary}
\theoremstyle{definition}
\newtheorem{definition}[theorem]{Definition}
\newtheorem{example}[theorem]{Example}
\newtheorem{remark}[theorem]{Remark}
\newcommand{\Jac}{\operatorname{Jac}}
\newcommand{\Spec}{\operatorname{Spec}}
\newcommand{\Max}{\operatorname{Max}}
\newcommand{\Frac}{\operatorname{Frac}}

\newcommand{\QQ}{\mathbb Q}
\newcommand{\ZZ}{\mathbb Z}
\newcommand{\FF}{\mathbb F}

\newcommand{\calQ}{\mathcal Q}
\newcommand{\calZ}{\mathcal Z}

\newcommand{\ep}{\varepsilon}
\setlist[enumerate]{itemsep=3pt,topsep=5pt}
\title{Every commutative perfect ring is a quotient\ of an almost perfect domain}
\author{Xiaolei Zhang \thanks{School of Mathematics and Statistics, Tianshui Normal University, Tianshui 741000, China; E-mail: zxlrghj@163.com}}
\date{}
\begin{document}
	\maketitle
	\begin{abstract}
In this paper, we show that every nonzero commutative perfect ring is a quotient of a
		one-dimensional semilocal almost perfect domain of characteristic zero.
		The domain can be chosen to have as many maximal ideals as the given ring
		has local factors.  The construction uses coefficient discrete
		valuation rings, triangular presentations of purely inseparable residue
		extensions, and approximation for finitely many independent rank-one
		valuations. Independent error variables keep the construction inside a
		field. A normal-form argument shows that the divided elements used to get
		perfect principal quotients do not kill any further good monomial. In
		positive residue characteristic, the choice $a=2c+1$ with $c\ge n$
		protects the nonzero powers of $p$ in a target ring of characteristic
		$p^n$. We also give local and mixed-characteristic consequences, and
		examples with non-nilpotent radicals.
	\end{abstract}
	
	
	\section{Introduction}
	
	Throughout this paper, all rings are commutative with identity and all modules are unital. All	choices in the proof are made in ZFC.

	Perfect rings were introduced by Bass in his work on projective covers
	and big finitistic dimension, that is a  ring $R$ is perfect if and only every $R$-module has a projective cover, if and only if its big finitistic dimension is $0$ \cite{Bass}. Moreover, a ring is
	perfect if and only if it is a finite product of local rings with
	$T$-nilpotent maximal ideals. Standard accounts of perfect rings can be found in
	\cite{AF,Lam}.
	
	Smith studied local domains with topologically $T$-nilpotent radical
	\cite{Smith}. Bazzoni and Salce developed the theory of almost perfect
	domains (abbreviates APDs) and its relation with strongly flat covers and big finitistic dimension, that is, an integral domain $D$ is almost perfect if and only if every $D$-module has a strongly flat cover if and only if its big finitistic dimension is at most $1$  
	\cite{BS02,BS03}. In particular, they proved the quotient criterion: an integral domain $D$ is an almost perfect domain  if $D/I$ is a perfect ring for
	every nonzero proper ideal $I$. This definition suggests the converse
	question which is proposed by Salce \cite{Salce}: 
	\begin{center}
		\textbf{Question 1.}  Is every perfect commutative ring $R$ the quotient of an APD?
	\end{center}
	
	The main motivation of this paper is to resolve this question for any perfect commutative rings. 	The realization problem has two separate difficulties. A local perfect
	ring need not have a coefficient field. Even in finite characteristic,
	its residue field may contain a purely inseparable extension that cannot
	be lifted as a field. Moreover, the product of domains realizing the
	local factors is not a domain. A construction must therefore keep both
	the residue extensions and the radical relations, while placing all
	local factors in a common field.
	
	Our main result is the following.
	
	\begin{theorem}\label{thm:main}
		Let $B$ be a nonzero commutative perfect ring, and let $r=|\Max(B)|$.
		There are a characteristic-zero, one-dimensional semilocal almost
		perfect domain $D$ and a nonzero proper ideal $I$ of $D$ such that
		\[
		D/I\cong B,\qquad |\Max(D)|=r.
		\]
		The domain can be chosen with an element $0\ne g\in I$ such that
		$D/gD$ is perfect and $g\in\Jac(D)$.
	\end{theorem}

	We sketch the argument before giving details. In positive residue
	characteristic $p$, a local perfect ring of characteristic $p^n$ admits
	a map from a characteristic-zero DVR with uniformizer $p$. Its residue
	image can be chosen so that the remaining residue extension is purely
	inseparable. We write this extension as a triangular tower. Each lifted
	equation has a radical error. Assigning an independent error variable
	to each equation gives a field extension with a standard monomial
	basis.
	
	Next we put the coefficient objects into a common algebraically closed
	field of characteristic zero, with finitely many independent rank-one
	valuations. Truncated valuation ideals give a one-dimensional semilocal
	coefficient domain. Approximation provides parameters that vanish at
	one chosen center and are congruent to $1$ at all other centers.
	Divided products of the zero monomials make each chosen principal
	quotient a nilpotent extension of a perfect monomial cover.
	
	The main point is that no unwanted relations appear. In positive residue
	characteristic we take
	\[
	p=\tau^{2c+1},\qquad \sigma=\tau^c,\qquad c\ge n.
	\]
	An occurrence count controls every coefficient that contains a
	truncated-valuation error. For the remaining terms, algebraic
	independence of the power variables separates the constant coefficient.
	Comparing exponents modulo $2c+1$ then protects every nonzero $p^vX^e$.
	This is done in Section~\ref{sec:normal}.
	
	Section~2 contains preliminaries. Sections~3 and 4 construct the
	coefficient objects and their common valued field. Sections~5 and 6
	give the generic equations and the integral domain. The local quotient
	calculation is in Section~7. We finish the proof in Section~8 and
	discuss examples in Section~9.
	
	\section{Perfect rings and principal quotients}
	
We give some basic facts and properties on Perfect rings and their quotients in this section.
	
	\begin{definition}
		An ideal $N$ of a ring $A$ is \emph{$T$-nilpotent} if every sequence
		$a_1,a_2,\ldots$ in $N$ has an initial product $a_1\cdots a_m=0$.
		A domain $D$ is \emph{almost perfect} if every nonzero proper
		quotient $D/I$ is perfect.
	\end{definition}
	
	We use the following form of Bass's theorem
	\cite[Theorem P]{Bass}; see also \cite[Theorem 1.1]{BS03} or \cite[Theorem 3.10.18]{WK}.
	
	\begin{theorem}[Bass]\label{thm:bass}
		For a commutative ring $A$, the following conditions are equivalent:
		\begin{enumerate}[(i)]
			\item $A$ is perfect;
			\item $A$ satisfies the descending chain condition on principal ideals;
			\item $A$ is a finite product of local rings with $T$-nilpotent
			maximal ideals;
			\item $A$ is semilocal and $\Jac(A)$ is $T$-nilpotent.
		\end{enumerate}
		Moreover, a semilocal ring is perfect if and only if all its maximal
		localizations are perfect.
	\end{theorem}

	\begin{lemma}\label{lem:lifting}
		Quotients of perfect rings are perfect. If $N$ is a nilpotent ideal
		of $A$ and $A/N$ is perfect, then $A$ is perfect.
	\end{lemma}
	\begin{proof}
		Semilocality and $T$-nilpotence of the radical pass to quotients.
		For the second claim, $N\subseteq\Jac(A)$ and $A$ is semilocal.
		Given a sequence in $\Jac(A)$, perfectness modulo $N$ gives a finite
		initial block whose product lies in $N$. Apply the same argument to
		the tail and repeat. If $N^s=0$, the product of $s$ such blocks is
		zero. Then Theorem~\ref{thm:bass} applies.
	\end{proof}
	
	\begin{lemma}\label{lem:extendT}
		If $N$ is a $T$-nilpotent ideal of $A$ and $S$ is a commutative
		$A$-algebra, then $NS$ is $T$-nilpotent.
	\end{lemma}
	\begin{proof}
		Assume $f_1f_2\cdots f_m\ne0$ for every $m$, where each $f_j$ is a
		finite sum $\sum_k a_{jk}s_{jk}$ with $a_{jk}\in N$. For each $m$,
		some choice of terms has nonzero product in $S$; its coefficient
		product in $A$ is then nonzero. These finite choices form a finitely
		branching tree, whose vertices are coefficient choices with nonzero
		product in $A$. Every level is nonempty. By K\"onig's lemma there is
		a sequence in $N$ all of whose initial products are nonzero, a
		contradiction.
	\end{proof}
	
	\begin{proposition}\label{prop:principal}
		Let $D$ be a one-dimensional semilocal domain. Suppose that
		$0\ne g\in\Jac(D)$ and $D/gD$ is perfect. Then $D$ is almost perfect.
	\end{proposition}
	\begin{proof}
		For every $m\ge1$ the ideal $gD/g^mD$ is nilpotent, so
		Lemma~\ref{lem:lifting} shows that $D/g^mD$ is perfect. Let $b$ be a
		nonzero nonunit of $D$. Every prime containing $bD$ is nonzero and
		hence maximal. Since $g$ lies in every maximal ideal, $g\in\sqrt{bD}$.
		Thus $g^mD\subseteq bD$ for some $m$, and $D/bD$ is a quotient of
		$D/g^mD$. Finally, every nonzero proper ideal contains a principal
		ideal of this form.
	\end{proof}
	
	In the local case, Proposition~\ref{prop:principal} is the usual
	one-principal-quotient criterion
	\cite[Proposition 2.5]{BS03}. The proof above gives the semilocal form
	we need.
	
	\begin{lemma}\label{lem:finiteNewton}
		Let $(A,\mathfrak m)$ be local, let $f\in A[X]$, and suppose
		$f(x)\in\mathfrak m$ is nilpotent and $f'(x)$ is a unit. Then there
		is $y\in A$ with $f(y)=0$ and $y-x\in\mathfrak m$.
	\end{lemma}
	\begin{proof}
		Set $x_0=x$ and $x_{j+1}=x_j-f(x_j)/f'(x_j)$. All $x_j$ have the same
		residue, so the denominators stay units. Taylor expansion gives
		$f(x_{j+1})\in f(x_j)^2A$. Hence $f(x_j)\in f(x_0)^{2^j}A$, which is
		zero for large $j$.
	\end{proof}
	
	\section{Coefficient objects and presentations}
	
	\begin{proposition}\label{prop:coefficients}
		Let $(B,\mathfrak m,k)$ be local perfect.
		\begin{enumerate}[(a)]
			\item If $\operatorname{char}k=0$, then $B$ contains a coefficient
			field isomorphic to $k$.
			\item If $\operatorname{char}k=p>0$, there are $n\ge1$, a
			characteristic-zero DVR $C$ with uniformizer $p$, and a local map
			$\eta:C\to B$ such that
			\[
			\operatorname{char}B=p^n,\qquad \ker\eta=p^nC,
			\]
			and $k/k_0$ is algebraic purely inseparable, where
			$k_0=C/pC$ is identified with its residue image in $k$.
		\end{enumerate}
		In both cases the fraction field of the coefficient object is algebraic
		over $\QQ(T)$ for some set $T$ of independent elements.
	\end{proposition}
	\begin{proof}
		Every element of $\mathfrak m$ is nilpotent. If $k$ has characteristic
		zero, every nonzero integer has nonzero residue and is a unit, so
		$\QQ\subseteq B$. Lift a transcendence basis $T$ of $k/\QQ$ to $B$.
		These lifts are algebraically independent, and every nonzero polynomial
		in them is a unit. Hence they generate an embedded field $\QQ(T)$.
		The extension $k/\QQ(T)$ is separable algebraic. By Zorn's lemma a
		maximal field lift exists. If its residue field is not $k$, lift the
		minimal polynomial of a missing separable element and apply
		Lemma~\ref{lem:finiteNewton}. Its root generates a larger embedded
		field, a contradiction. This proves (a).
		
		Suppose $k$ has characteristic $p$. The element $p1_B$ is nilpotent,
		so $B$ has characteristic exactly $p^n$ for some $n\ge1$. Choose a
		transcendence basis $T$ of $k/\FF_p$ and lift it to $B$. Evaluation
		gives a local map
		\[
		C_0=\ZZ[T]_{(p)}\longrightarrow B.
		\]
		Here we localize at the prime generated by $p$, even when $T$ is
		infinite. Every nonzero element of $C_0$ is $p^ru$ with $u$ a unit, so
		$C_0$ is a DVR with residue field $\FF_p(T)$.
		
		Let $k_0$ be the maximal separable subextension of
		$k/\FF_p(T)$. Every element of $k$ has a $p$-power that is separable
		over $\FF_p(T)$; this follows by removing powers of $X^p$ from its
		irreducible polynomial. Hence $k/k_0$ is purely inseparable.
		
		Starting from $C_0$, we lift the separable algebraic extension $k_0$.
		At a successor step, suppose $C'$ already maps locally to $B$, and let
		$\alpha$ be a separable algebraic residue element not yet lifted. Lift
		its monic irreducible polynomial to $f\in C'[X]$. Any lift $x\in B$
		has nilpotent error $f(x)$ and unit derivative.
		Lemma~\ref{lem:finiteNewton} gives a root $y\in B$ with the required
		residue. The ring $C'[X]/(f)$ is a DVR with uniformizer $p$: it is a
		finite domain over $C'$, and its maximal ideal is generated by $p$.
		Indeed, irreducibility modulo $p$ implies irreducibility over the
		fraction field, and the quotient modulo $p$ is the prescribed residue
		field. The map sending $X$ to $y$ is local.
		
		At a limit step take the union. Every nonzero element of the union is
		still a power of $p$ times a unit, so the union is a DVR with the same
		uniformizer. Well-order a generating set of $k_0$ and repeat the step
		to finish the construction. The kernel of the map to $B$ is exactly
		$p^nC$, since units map to units and $p^r1_B\ne0$ for $r<n$. All
		successive fraction-field extensions are algebraic over the initial
		$\QQ(T)$, which also proves the last assertion.
	\end{proof}
	
	This proposition is a coefficient construction over a nil local
	extension. It uses only finite correction of individual errors and no
	completeness of $B$. For comparison, the usual Cohen structure theorem
	for complete local rings is treated in \cite{Matsumura,Stacks}.
	
	\begin{lemma}\label{lem:puretower}
		If $k/k_0$ is algebraic purely inseparable, it has a well-ordered
		generating family $(\xi_\alpha)$ with relations
		\[
		\xi_\alpha^{d_\alpha}=f_\alpha(\xi_\beta\mid\beta<\alpha),
		\qquad d_\alpha=p^{e_\alpha}\ge2,
		\]
		where each $f_\alpha$ is a polynomial over $k_0$ involving finitely
		many earlier generators. The monomials $\prod\xi_\alpha^{h_\alpha}$,
		$0\le h_\alpha<d_\alpha$, form a $k_0$-basis of $k$.
	\end{lemma}
	\begin{proof}
		Well-order a generating set and omit elements already in the field
		generated by their predecessors. At each step take the smallest
		$p$-power lying in the preceding field. Its degree is that $p$-power.
		A field generated by algebraic elements over a field is already
		generated by them as an algebra: inverses lie in the finite
		subextensions in which they occur. Thus the coefficient on the right is
		a polynomial in finitely many earlier generators. The basis assertion
		follows successively from the extension degrees and then by taking
		unions.
	\end{proof}
	
	From now on fix a Bass decomposition $B=\prod_{i=1}^rB_i$.
	For positive residue characteristic let $C_i,k_{i,0},n_i$ be as in
	Proposition~\ref{prop:coefficients}; for zero residue characteristic
	put $C_i=k_{i,0}=k_i$. Let $\Omega_i$ index the tower in
	Lemma~\ref{lem:puretower} when the residue characteristic is
	positive, and let $\Omega_i$ be empty otherwise. A finite
	concatenation makes $\Omega=\coprod_i\Omega_i$ a well-ordered set
	respecting each tower.
	
	\section{A common valued coefficient domain}\label{sec:base}
	
	We give the valued-field construction explicitly, since the local
	coefficient fields may have different cardinalities and residue
	characteristics.
	
	\begin{lemma}\label{lem:commonfield}
		There are an algebraically closed characteristic-zero field $F$,
		embeddings of all $\Frac(C_i)$ into $F$, and distinct rank-one
		valuations $v_i$ on $F$ such that $v_i$ restricts to the given DVR
		valuation on $C_i$ in positive residue characteristic and is trivial
		on $C_i$ in zero residue characteristic. The value groups can be
		identified with $\QQ$; in the positive case $v_i(p_i)=1$.
	\end{lemma}
	\begin{proof}
		Write $F_i=\Frac(C_i)$, interpreted as $C_i$ in the field case. By
		Proposition~\ref{prop:coefficients}, $F_i$ is algebraic over
		$\QQ(T_i)$. Make the sets $T_i$ disjoint and independent, add one
		variable $Z$, and take an algebraic closure of
		$\QQ(\coprod_iT_i,Z)$. Embed each $F_i$ over its own rational field.
		
		For a positive-characteristic index, start with the DVR valuation on
		$F_i$, extend it by the Gauss rule on the other independent $T_j$,
		and use a Gauss valuation in $Z-a_i$ with a fixed positive integer $M$
		as value. For all indices with the same prime $p$, choose distinct
		integers $a_i$ modulo $p^M$, with $M$ large enough. The restrictions
		to $\QQ(Z)$ are distinct: $v_i(Z-a_i)=M$, while $v_j(Z-a_i)<M$ for
		another center with the same prime. Different primes are distinguished
		by their values on integers. In the field case use a $(Z-b_i)$-adic
		valuation, trivial on $F_i$ and on the other independent variables,
		with distinct rational $b_i$. These valuations are distinct from the
		positive-residue-characteristic ones and from each other.
		
		Extend each valuation to the algebraic closure $F$. An algebraic
		extension of a discrete rank-one valuation has rational value group;
		in the algebraic closure this group is divisible, hence is $\QQ$ with
		the stated normalization. Distinct restrictions stay distinct after
		extension. These rank-one valuations are independent. We use the usual
		extension and approximation theorems for valuations; see
		\cite[Chapter VI]{ZS} and \cite[Section 1]{Vaquie}.
	\end{proof}
	
	For a positive index choose integers
	\begin{equation}\label{eq:ac}
		c_i\ge n_i,\qquad a_i=2c_i+1,
		\qquad \tau_i^{a_i}=p_i,\qquad \sigma_i=\tau_i^{c_i}
	\end{equation}
	in $F$. For a zero-characteristic index choose $\sigma_i\in F$
	with $v_i(\sigma_i)=1$. Define
	\[
	H_i=\begin{cases}n_i,&p_i>0,\\2,&p_i=0,\end{cases}
	\qquad A_i^*=\{z\in F:v_i(z)\ge H_i\}.
	\]
	Here $p_i=0$ is a label for the zero-residue-characteristic case.
	Let $V_i^*$ be the valuation ring and put
	\[
	V^*=\bigcap_iV_i^*,\qquad
	R^*=\bigcap_i(C_i+A_i^*),\qquad
	A^*=\bigcap_i A_i^*.
	\]
	The ideals $A_i^*$ are understood as ideals of $V_i^*$.
	
	\begin{lemma}\label{lem:pullback}
		The ring $R^*$ is a one-dimensional semilocal domain with maximal
		ideals $\mathfrak p_1^*,\ldots,\mathfrak p_r^*$. Its localizations
		and coefficient quotient are
		\begin{align}
			R^*_{\mathfrak p_i^*}&=C_i+A_i^*,\label{eq:localbase}\\
			R^*/A^*&\cong\prod_i\overline C_i,\label{eq:basequot}
		\end{align}
		where $\overline C_i=C_i/p_i^{n_i}C_i$ for $p_i>0$ and
		$\overline C_i=C_i$ for $p_i=0$.
	\end{lemma}
	\begin{proof}
		Approximation for the independent valuations gives
		\[
		V^*/A^*\cong\prod_i V_i^*/A_i^*.
		\]
		The inverse image of $\prod_i\overline C_i$ is exactly $R^*$, so
		\eqref{eq:basequot} follows. In the positive case
		$C_i\cap A_i^*=p_i^{n_i}C_i$; in the field case the intersection is
		zero.
		
		The intersection $V^*$ is a one-dimensional semilocal Pr\"ufer
		domain, and its maximal localizations are the $V_i^*$. This again
		follows from approximation: one can clear poles at all other centers
		without changing a residue at a chosen center. Its nonzero primes are
		its $r$ maximal ideals. The ideal $A^*$ is a nonzero ideal of $V^*$
		contained in all these maximal ideals, and it is a conductor ideal
		into $R^*$.
		
		Any prime of $R^*$ not containing $A^*$ is zero. Indeed, choose
		$0\ne a\in A^*$ outside that prime. Then
		$R^*[1/a]=V^*[1/a]=F$. Primes containing $A^*$ correspond, by
		\eqref{eq:basequot}, to the $r$ maximal ideals of the displayed
		zero-dimensional ring. This proves dimension and semilocality.
		
		For completeness, let $z\in C_i+A_i^*$. Approximation gives $h$
		congruent to $1$ modulo $A_i^*$, with value at all other centers so
		large that both $h$ and $hz$ belong to $R^*$. Then
		$h\notin\mathfrak p_i^*$ and $z=(hz)/h$. This proves
		\eqref{eq:localbase}.
	\end{proof}
	
	Choose $e_i\in R^*$ mapping in \eqref{eq:basequot} to $1$ in the
	$i$-th coordinate and to $0$ elsewhere. Choose every later lift of a
	coefficient in $\overline C_i$ in the same way, with zero in the other
	coordinates. Composing \eqref{eq:basequot} with the fixed coefficient
	maps gives a homomorphism
	\begin{equation}\label{eq:baseeval}
		R^*\longrightarrow B.
	\end{equation}
	
	Approximation also gives nonzero $s_1,\ldots,s_r,w\in F$ such that
	\begin{align}
		s_i/\sigma_i&\in1+A_i^*,&
		s_i&\in1+A_j^*\quad(j\ne i),\label{eq:sapprox}\\
		w/p_i&\in1+A_i^*\quad(p_i>0),&
		w/\sigma_i^2&\in1+A_i^*\quad(p_i=0).\label{eq:wapprox}
	\end{align}
	To get a ratio congruence, it suffices to approximate the numerator
	with precision $H_i$ plus the value of its denominator. Note that
	$w\in R^*$: in positive residue characteristic it is congruent to
	$p_i$ modulo $A_i^*$, and in zero residue characteristic it lies in
	$A_i^*$.
	
	We extend these valuations by the Gauss rule to a rational field
	$K=F(U_\lambda)$. Write $V_i,A_i$ for the extended valuation ring and
	the ideal of values at least $H_i$, and set
	\begin{equation}\label{eq:R}
		R=\bigcap_i(C_i+A_i)\subseteq K.
	\end{equation}
	Lemma~\ref{lem:pullback} holds verbatim for $R$, and $R^*\subseteq R$.
	In particular, all coefficient choices already made in $R^*$ remain
	valid. This order of choices ensures that they are independent of all
	variables $U_\lambda$ and $Y_\lambda$ introduced below.
	
	\section{Independent error variables}
	
	For each positive-characteristic index, lift the polynomials in
	Lemma~\ref{lem:puretower} to $C_i$ and choose lifts $u_{i,\alpha}$
	of the residue generators in $B_i$. We regard them as elements of $B$
	with zero coordinates outside $i$.
	
	Lift every coefficient of the resulting polynomial to $R^*$ with
	support only in coordinate $i$, as specified after
	Lemma~\ref{lem:pullback}. Let $F_\alpha$ denote the lifted polynomial
	in the earlier indeterminates. Define
	\begin{equation}\label{eq:genericP}
		P_\alpha(Z)=Z^{d_\alpha}-(1-e_i)Z
		-F_\alpha(Z_\beta\mid\beta<\alpha),
		\qquad \alpha\in\Omega_i.
	\end{equation}
	Modulo $A_i$ this is the desired lifted purely inseparable equation.
	Modulo $A_j$ for $j\ne i$ it is $Z^{d_\alpha}-Z$.
	Evaluation in $B$ gives an error
	\[
	b_{\lambda(\alpha)}=P_\alpha(u_{i,\alpha})\in\Jac(B)
	\]
	supported only in factor $i$. Assign a new index $\lambda(\alpha)$ to
	each equation, even when errors are equal or zero.
	
	Let $(b_\lambda)_{\lambda\in\Lambda}$ consist of all these fresh
	errors and a family containing every element of each
	$\mathfrak m_i$, supported only in its own factor. Let
	$b_{i,\lambda}$ be its $i$-th coordinate. Choose
	\begin{equation}\label{eq:N}
		N_\lambda\ge2,\qquad b_\lambda^{N_\lambda}=0,
		\qquad N_\lambda\ge n_i\ \text{for every }p_i>0.
	\end{equation}
	For a zero error the first exponent condition is automatic.
	
	Take $Y_\lambda$ algebraically independent over $F$ and set
	\begin{equation}\label{eq:fields}
		U_\lambda=Y_\lambda^{N_\lambda},\quad
		K=F(U_\lambda\mid\lambda\in\Lambda),\quad
		L=F(Y_\lambda\mid\lambda\in\Lambda),\quad
		x_\lambda=wY_\lambda.
	\end{equation}
	Use this $K$ in \eqref{eq:R}. Extend each $v_i$ to $L$ by the Gauss
	rule on the $Y_\lambda$. The residue variables $Y_\lambda$ are
	independent over the residue field of $F$.
	
	In an algebraic closure of $L$, choose successively elements
	$z_\alpha$ satisfying
	\begin{equation}\label{eq:z}
		P_\alpha(z_\alpha)=x_{\lambda(\alpha)}.
	\end{equation}
	Let $E=L(z_\alpha\mid\alpha\in\Omega)$.
	
	\begin{lemma}\label{lem:genericbasis}
		The monomials
		\begin{equation}\label{eq:basis}
			Y^e z^h,\qquad
			0\le e_\lambda<N_\lambda,\quad
			0\le h_\alpha<d_\alpha,
		\end{equation}
		with both exponent vectors of finite support, form a $K$-basis of $E$.
		Each $v_i$ can be extended to $E$ so that the residue of $z_\alpha$
		is $\xi_{i,\alpha}$ when $\alpha\in\Omega_i$, and is $0$ otherwise.
	\end{lemma}
	\begin{proof}
		The monomials $Y^e$ form a $K$-basis of $L$: take finite subsets of
		the independent variables and use the power relations
		$Y_\lambda^{N_\lambda}=U_\lambda$.
		
		At stage $\alpha$, all earlier $z_\beta$ are algebraic over the field
		generated by $F$ and the $Y$'s other than $Y_{\lambda(\alpha)}$.
		Their defining equations never used that fresh variable. Let $G$ be
		that field with all earlier $z_\beta$ adjoined. The preceding field in
		the tower is $G(Y_{\lambda(\alpha)})$. The polynomial
		\[
		P_\alpha(Z)-wY_{\lambda(\alpha)}
		\]
		is irreducible there and has degree $d_\alpha$. Indeed, for any field
		$G$ and a degree-$d$ polynomial $P$, the rational function extension
		$G(Z)/G(P(Z))$ has degree $d$. We apply this to the indeterminate
		$wY_{\lambda(\alpha)}$. To check the degree assertion, suppose
		$1,Z,\ldots,Z^{d-1}$ were dependent over $G(P(Z))$. Clearing
		denominators would give
		$\sum_{j=0}^{d-1} A_j(P(Z))Z^j=0$, with $A_j\in G[T]$ not all
		zero. The nonzero summands have distinct degrees $d\deg A_j+j$. The
		largest-degree terms cannot cancel. Hence these $d$ elements are
		independent; the equation defining $P(Z)$ gives the opposite bound.
		Thus the standard powers of each $z_\alpha$ form a basis at that step.
		Combining the two towers proves \eqref{eq:basis}.
		
		The valuation extensions can be chosen at the same time. All
		coefficients of each monic equation have nonnegative value, and
		$x_{\lambda(\alpha)}$ has positive value. At the center belonging to
		$\alpha$, the residue equation is the purely inseparable equation in
		\eqref{eq:genericP}, and its prescribed residue root exists. At
		another center the residue equation is $Z^{d_\alpha}-Z$, and $0$ is a
		simple root. A prime above the valuation maximal ideal with this
		chosen residue root gives an extension of the valuation to the root
		field. This is the usual extension of a place through an algebraic
		extension \cite[Chapter VI, Section 11]{ZS}. Taking unions gives
		compatible extensions on $E$.
	\end{proof}
	
	\begin{remark}\label{rem:finiteclosure}
		All normal-form arguments below take place in finite systems. A finite
		set of indices has a finite closure under the earlier indices occurring
		in their equations. To see this, unfold the dependencies as a rooted
		tree. The tree is finitely branching. An infinite branch would give a
		strictly decreasing sequence of well-ordered indices, which is
		impossible. By K\"onig's lemma the tree is finite. In a finite closed
		system, reducing the largest index first gives a terminating monic
		division procedure. The lower term $(1-e_i)Z$ in \eqref{eq:genericP}
		also has degree smaller than $d_\alpha$.
	\end{remark}
	
	\section{The integral domain construction}\label{sec:construction}
	
	For $p_i>0$, call a pair $(r,\nu)$ \emph{$i$-bad} if
	\[
	p_i^r b_i^\nu=0\quad\text{in }B_i,
	\qquad r\ge0.
	\]
	For $p_i=0$, use only pairs $(0,\nu)$ with $b_i^\nu=0$.
	No bad pair has $r+|\nu|=0$.
	
	For $\ell\ge1$ and $i$-bad pairs
	$(r_1,\nu^{(1)}),\ldots,(r_\ell,\nu^{(\ell)})$, put
	\begin{equation}\label{eq:q}
		q_i=s_i^{-(2\ell-1)}p_i^{r_1+\cdots+r_\ell}
		x^{\nu^{(1)}+\cdots+\nu^{(\ell)}}.
	\end{equation}
	In the zero-characteristic case we omit the $p_i$ factor. Let
	$\calQ_i$ be the set of all such elements, and define
	\begin{equation}\label{eq:D0}
		D_0=R[s_1,\ldots,s_r,x_\lambda,z_\alpha,
		\calQ_1,\ldots,\calQ_r]\subseteq E.
	\end{equation}
	Concatenating the lists of bad pairs gives the exact relation
	\begin{equation}\label{eq:Qsquare}
		q,q'\in\calQ_i\quad\Longrightarrow\quad qq'=s_iq''
		\quad\text{for some }q''\in\calQ_i.
	\end{equation}
	
	\begin{lemma}\label{lem:integral}
		The domain $D_0$ is integral over $R$ and has dimension one.
	\end{lemma}
	\begin{proof}
		The values of $w$ are $1$ at positive-characteristic centers and $2$
		at zero-characteristic centers. Hence
		$x_\lambda^{N_\lambda}=w^{N_\lambda}U_\lambda$ lies in every $A_i$,
		by \eqref{eq:N}, and therefore in $R$; this proves that $x_\lambda$
		is integral. A sufficiently large power of $s_i$ lies in $A_i$ at its
		own center and is congruent to $1$ modulo $A_j$ at all other centers,
		so that power belongs to $R$. The $z_\alpha$ are integral
		successively by their monic equations and transitivity of integrality.
		
		Consider \eqref{eq:q}, and put $r'=\sum r_j$, $\mu=\sum\nu^{(j)}$.
		At a positive-characteristic center $i$,
		\[
		v_i(q_i)=r'+|\mu|-(2\ell-1)c_i/a_i
		\ge \ell-(2\ell-1)c_i/a_i=(\ell+c_i)/a_i>0.
		\]
		At a zero-characteristic center $i$, $|\mu|\ge\ell$ and the value is
		at least $2\ell-(2\ell-1)=1$. Take a power that removes all $Y$
		exponents using the $N_\lambda$. Increase it until the value at center
		$i$ is at least $H_i$. At a different center $j$ the denominator $s_i$
		is congruent to $1$ modulo $A_j$. If $|\mu|>0$, large powers of the
		numerator have value at least $H_j$. If $\mu=0$, the numerator is an
		integer $p_i^{r'}$, and every power of $q_i$ is congruent modulo $A_j$
		to the corresponding power of that integer, which lies in $C_j$.
		Hence one common power of $q_i$ belongs to every $C_j+A_j$, and so to
		$R$. This proves integrality. All generators in \eqref{eq:D0} are
		integral. Dimension does not change under an integral extension, so
		$\dim D_0=\dim R=1$.
	\end{proof}
	
	The extended valuations are nonnegative on $D_0$. Their residue
	images are exactly $k_i$: the coefficient image is $k_{i,0}$, the
	native $z_\alpha$ give its purely inseparable generators, the foreign
	$z_\alpha$ give zero, and the remaining generators have residues in
	this field. In particular, $s_i,x_\lambda,\calQ_i$ have residue zero
	at center $i$, while the other $s_j$ have residue one. A foreign
	divided element has residue zero if it contains an $x$; otherwise its
	residue is the indicated power of an integer.
	
	Let $\mathfrak n_i$ be the kernel of the residue map $D_0\to k_i$.
	These are distinct maximal ideals, since $s_i$ lies in $\mathfrak n_i$
	and not in the others. Set
	\begin{equation}\label{eq:D}
		D=(D_0)_{D_0\setminus\bigcup_i\mathfrak n_i}.
	\end{equation}
	This is a one-dimensional semilocal domain with exactly these maximal
	ideals, extended to $D$. Each is nonzero because it contains
	$s_i\ne0$.
	
	\section{The local quotient and its normal form}\label{sec:normal}
	
	Fix $i$, and put $S_i=D_0R_{\mathfrak p_i}$. In this section all
	localizations of $R$ are taken inside $K$. Write $R_i=C_i+A_i$.
	
	At center $i$, all foreign $s_j$ are units of $R_i$, and their divided
	elements lie in $R_i[x_\lambda]$. We may replace $s_i$ by $\sigma_i$
	and $x_\lambda$ by
	\begin{equation}\label{eq:normalizedx}
		\widehat x_\lambda=
		\begin{cases}p_iY_\lambda,&p_i>0,\\
			\sigma_i^2Y_\lambda,&p_i=0.
		\end{cases}
	\end{equation}
	The ratios in \eqref{eq:sapprox} and \eqref{eq:wapprox} are units of
	$R_i$, congruent to $1$ modulo $A_i$. Hence the normalized divided
	elements, defined by \eqref{eq:q} with $\sigma_i$ and $\widehat x$,
	generate the same $R_i$-algebra. Denote their set by
	$\widehat{\calQ}_i$. From now on we write $x$ for $\widehat x$ in the
	local calculation. Define
	\begin{equation}\label{eq:J}
		J_i=(A_i,\sigma_i,\widehat{\calQ}_i)S_i.
	\end{equation}
	
	Let $\calZ_i$ be the ideal of $C_i[X_\lambda]$ generated by the
	$i$-bad monomials, with coefficients $p_i^r$ in the positive case, and
	put
	\[
	P_i=C_i[X_\lambda]/\calZ_i.
	\]
	For $p_i>0$ this ideal contains $p_i^{n_i}$; in both cases it contains
	$X_\lambda^{N_\lambda}$ for every $\lambda$.
	
	Take the $C_i$ parts of the coefficients of the normalized equations.
	They give the following monic algebra over $P_i$:
	\begin{equation}\label{eq:Ti}
		T_i=P_i[Z_\alpha\mid\alpha\in\Omega]/
		\left(
		\begin{array}{ll}
			Z_\alpha^{d_\alpha}-\widetilde f_{i,\alpha}(Z_\beta)-X_{\lambda(\alpha)},
			&\alpha\in\Omega_i,\\
			Z_\alpha^{d_\alpha}-Z_\alpha-X_{\lambda(\alpha)},
			&\alpha\notin\Omega_i
		\end{array}\right).
	\end{equation}
	For a foreign index, $b_{i,\lambda(\alpha)}=0$, so its last $X$ term
	is already zero in $P_i$. In a positive-characteristic coefficient
	ring, choices differing by $p_i^{n_i}C_i$ give the same relations. All
	differences between the actual normalized coefficients in $R_i$ and
	the displayed $C_i$ coefficients lie in $A_i$.
	
	There is a surjection
	\begin{equation}\label{eq:surjTi}
		T_i\longrightarrow S_i/J_i,
		\qquad X_\lambda\longmapsto x_\lambda+J_i,
		\quad Z_\alpha\longmapsto z_\alpha+J_i.
	\end{equation}
	Indeed, every bad monomial is $\sigma_i$ times its first divided
	element, hence zero in this quotient. The remaining generators and
	equations have just been described.
	
	\begin{lemma}\label{lem:abstractbasis}
		The algebra $T_i$ is free as a $P_i$-module on the standard monomials
		$Z^h$, $0\le h_\alpha<d_\alpha$. In positive residue characteristic,
		it is the direct sum, over standard $h$ and good $e$, of the
		coefficient modules
		\[
		(C_i/p_i^{n_e}C_i)X^eZ^h,\qquad
		n_e=\min\{v\ge0:p_i^vb_i^e=0\}.
		\]
		Here good $e$ means $b_i^e\ne0$, so $1\le n_e\le n_i$. In zero
		residue characteristic its good $X^eZ^h$ are a $C_i$-basis.
	\end{lemma}
	\begin{proof}
		Adjoining one variable with a monic degree-$d$ relation gives a free
		module on $1,Z,\ldots,Z^{d-1}$. Apply this successively, using finite
		closed systems as in Remark~\ref{rem:finiteclosure}. The coefficient
		module of a fixed $X^e$ in $P_i$ is exactly as stated, since the
		defining coefficient-monomial ideal has no relations between different
		$X$ monomials. A good $e$ necessarily satisfies $e_\lambda<N_\lambda$.
	\end{proof}
	
	\begin{lemma}[Occurrence count]\label{lem:count}
		Suppose a term contains $\ell$ bad monomial factors. During monic
		reduction of the $Z$ powers, variables and nonnegative powers of $p_i$
		may be added, but no original $X$ or $p_i$ factor is removed. After
		this reduction, let $r$ be the total $p_i$ exponent before reducing
		$X_\lambda^{N_\lambda}$, and let $\rho$ be the number of $X$
		occurrences removed in that reduction. If the final $p_i^vX^e$ is
		good, then
		\begin{equation}\label{eq:count}
			\ell\le\rho+\max(r-v,0).
		\end{equation}
		In particular, if $X^e$ is good, then $\ell\le\rho+r$. In zero
		residue characteristic, $\ell\le\rho$.
	\end{lemma}
	\begin{proof}
		Label the occurrences in each original bad factor. Remove the $\rho$
		occurrences prescribed by the $X$-power reductions. If $r>v$, also
		remove $r-v$ occurrences of $p_i$; otherwise remove none. At most $v$
		original $p_i$ occurrences remain. If fewer than $\ell$ occurrences
		in total have been removed, some original bad group is untouched. That
		bad monomial divides $p_i^vX^e$, a contradiction. Added factors do not
		affect the argument. The zero-characteristic case has no $p_i$ factors.
	\end{proof}
	
	\begin{theorem}[Local normal form]\label{thm:normal}
		The map \eqref{eq:surjTi} is an isomorphism.
	\end{theorem}
	\begin{proof}
		By Lemma~\ref{lem:abstractbasis}, a nonzero element of $T_i$ can be
		represented by a finite sum of good $X^eZ^h$ with coefficients in
		$C_i$, nonzero in their indicated coefficient modules. Its image in
		$E$ is a sum of the corresponding $x^ez^h$. These are part of the
		$K$-basis in Lemma~\ref{lem:genericbasis}, multiplied by nonzero
		scalars. We compare the coefficients separately.
		
		An expression in $J_i$ is a finite sum of terms containing a
		coefficient from $A_i$, a positive power of $\sigma_i$, or at least
		one normalized divided factor. Write every $R_i$ coefficient as a
		$C_i$ part plus an $A_i$ part. Reduce the $Z$ powers using their actual
		monic equations in $E$. These equations have exactly the displayed
		$C_i$ parts in \eqref{eq:Ti}; all other scalar coefficients lie in
		$A_i$. The expansion is finite. Each term still contains the monomial
		numerator of each original divided factor; monic reduction only adds
		$X$ factors or nonnegative $p_i$ powers. Then reduce the $Y$ exponents
		to the bounds $N_\lambda$. We now make the coefficient comparison
		precise.
		
		\smallskip\noindent
		\emph{Positive residue characteristic.}
		Suppress $i$, so $p=\tau^a$, $\sigma=\tau^c$, $a=2c+1$ and $c\ge n$.
		Suppose a term has $h$ divided factors involving $\ell$ bad leaves in
		total, plus an extra factor $\sigma^j$ with $j\ge0$. Their combined
		denominator is $\sigma^{2\ell-h}$. If there are no divided factors,
		take $h=\ell=0$. Write the final $X$ exponent before power reduction
		as
		\[
		\mu_\lambda=m_\lambda N_\lambda+e_\lambda,
		\qquad 0\le e_\lambda<N_\lambda,
		\qquad \rho=\sum_\lambda m_\lambda N_\lambda.
		\]
		Since $x_\lambda=pY_\lambda$, power reduction gives
		\[
		x^\mu=p^\rho U^m x^e.
		\]
		After absorbing any $p$-power from a nonzero $C$ coefficient into $r$,
		a coefficient of a good $x^ez^h$ has the form
		\begin{equation}\label{eq:coefficient}
			\gamma\,p^{r+\rho}\sigma^{j-2\ell+h}U^m,
			\qquad r,j\ge0,
		\end{equation}
		where $\gamma$ is a unit of $C$ or lies in $A$. The exponent $h$ here
		counts divided factors only; the standard $z$ exponent plays no role in
		the estimate.
		
		If $\gamma\in A$, its value in units of $v(\tau)$ is at least $an$.
		Since $X^e$ is good, Lemma~\ref{lem:count} gives $\ell\le r+\rho$.
		Hence the value of \eqref{eq:coefficient} is at least
		\begin{equation}\label{eq:higherror}
			an+a(r+\rho)+c(j-2\ell+h)
			\ge an+(a-2c)(r+\rho)+c(j+h)\ge an.
		\end{equation}
		Therefore all $A$-coefficient errors vanish when coefficients of a
		good $e$ are compared modulo value $an_e$, since $n_e\le n$.
		
		Now suppose $\gamma$ is a unit of $C$. The same estimate without $an$
		shows that the $\tau$ exponent is nonnegative. These terms lie in
		$C[\tau][U_\lambda]$. The map
		\begin{equation}\label{eq:injecttrunc}
			\bigl(C[\tau]/\tau^{an_e}C[\tau]\bigr)[U_\lambda]
			\longrightarrow V_i/\{z:v_i(z)\ge n_e\}
		\end{equation}
		is injective. To see this, note that $C[\tau]$ is the Eisenstein DVR
		with uniformizer $\tau$ and $\tau^a=p$. Its terms with distinct $\tau$
		exponents modulo $a$ have distinct fractional values. After terms of
		equal value are combined, independence of the residue variables
		$U_\lambda$ prevents further cancellation. Hence the value of a
		nonzero polynomial is the minimum of its coefficient values. This
		proves \eqref{eq:injecttrunc}.
		
		It follows that we can extract separately the constant $U$ coefficient
		and its $\tau$-exponent class $0$ modulo $a$. Such a constant $U$
		contribution requires $m=0$, hence $\rho=0$. If its $\tau$ exponent is
		$av$, then
		\[
		ar+c(j-2\ell+h)=av.
		\]
		Since $\gcd(a,c)=1$, there is an integer $k$ with
		\begin{equation}\label{eq:congruence}
			j-2\ell+h=ak,\qquad v=r+ck.
		\end{equation}
		First assume a divided factor is present, so $h\ge1$ and $\ell\ge1$.
		If $k\ge0$ and $p^vX^e$ were good, then \eqref{eq:count}, with
		$\rho=0$ and $v\ge r$, would give $\ell\le0$. If $k<0$, then
		\[
		a(-k)=2\ell-h-j<2\ell,
		\qquad r-v=c(-k)<\ell.
		\]
		Again \eqref{eq:count} rules out goodness of $p^vX^e$. Therefore every
		such constant contribution with $v<n_e$ is absent.
		
		If no divided factor is present, $h=\ell=0$. A term of $J_i$ without
		an $A$ coefficient must have $j\ge1$. In class $0$ modulo $a$,
		\eqref{eq:congruence} gives $j=ak$ with $k\ge1$ and
		$v=r+ck\ge c\ge n\ge n_e$. This contribution is also zero in the
		truncated coefficient module.
		
		We have shown that, for each good $e$ and standard $z$ monomial, every
		coefficient coming from $J_i$ has zero constant-$U$, class-$0$
		component modulo $p^{n_e}$. A nonzero coefficient chosen from
		$C/p^{n_e}C$ has a nonzero component there. The independent basis
		\eqref{eq:basis} therefore prevents a nonzero normal form from lying
		in $J_i$.
		
		\smallskip\noindent
		\emph{Zero residue characteristic.}
		Here $C=k_i$, $x_\lambda=\sigma^2Y_\lambda$, and there are no $p$
		factors. Lemma~\ref{lem:count} gives $\rho\ge\ell$ for a good $e$. A
		term with a $C$ coefficient then has scalar factor
		\[
		\sigma^{2\rho+j-2\ell+h}U^m,
		\qquad 2\rho+j-2\ell+h\ge j+h.
		\]
		For a term of $J_i$ without an $A_i$ coefficient, $j+h>0$. An $A_i$
		coefficient already has positive value, and the same inequality
		preserves it. Hence every coefficient of a good normal form coming
		from $J_i$ has positive value. A nonzero coefficient from $k_i$ has
		value zero. Basis independence again gives the required contradiction.
		This proves injectivity in both cases.
	\end{proof}
	
	\begin{remark}
		The two parts of the positive-characteristic calculation serve
		different purposes. Inequality \eqref{eq:higherror} handles arbitrary
		rational $U$ dependence in the truncated valuation ideal. Only after
		removing those terms do we extract a polynomial $U$ coefficient. The
		congruence \eqref{eq:congruence} then protects the remaining powers of
		$p$. It would not be justified to extract a constant coefficient
		before removing the valuation errors.
	\end{remark}
	
	\section{Perfect local quotients and the realization theorem}
	
	\begin{lemma}\label{lem:Pi}
		Each $P_i$ is local perfect. The algebra $T_i$, localized at the
		residue map specified by Lemma~\ref{lem:genericbasis}, is also local
		perfect and maps onto $B_i$.
	\end{lemma}
	\begin{proof}
		Let $N_i$ be the ideal of $P_i$ generated by $p_i,X_\lambda$ in the
		positive case, and by the $X_\lambda$ in the field case. Its quotient
		is $k_{i,0}$. To prove $T$-nilpotence, suppose a sequence in $N_i$ has
		all initial products nonzero. Write its terms as finite sums of
		nonconstant monomials in these generators, with unit coefficients. At
		each level some monomial product survives the defining
		coefficient-monomial ideal. The resulting finitely branching tree has
		an infinite branch. Evaluation in $B_i$ gives a sequence of radical
		monomials all of whose initial products are nonzero, contradicting
		$T$-nilpotence of $\mathfrak m_i$. Thus $N_i$ is $T$-nilpotent. It is
		nil, and its field quotient shows that it is the unique maximal ideal.
		Hence $P_i$ is perfect.
		
		Now consider only the native variables $Z_\alpha$, $\alpha\in\Omega_i$,
		and their equations in $T_i$, and denote this algebra by $T_i^0$.
		Modulo $N_iT_i^0$ it is exactly the field tower $k_i$ in
		Lemma~\ref{lem:puretower}. Lemma~\ref{lem:extendT} makes $N_iT_i^0$
		$T$-nilpotent. Therefore $T_i^0$ is local perfect.
		
		For a foreign index, its error variable is zero in $P_i$, so its
		equation is $Z_\alpha^{d_\alpha}-Z_\alpha=0$. At the selected residue
		root $0$, $Z_\alpha^{d_\alpha-1}-1$ is a unit, hence $Z_\alpha=0$.
		This holds for all foreign indices. Localizing $T_i$ at the selected
		maximal ideal therefore gives exactly $T_i^0$: after foreign variables
		are set to zero, every inverted element has nonzero residue in the
		already local $T_i^0$.
		
		Finally, evaluate coefficients using $\eta_i$, send every $X_\lambda$
		to $b_{i,\lambda}$, every native $Z_\alpha$ to its chosen
		$u_{i,\alpha}$, and every foreign $Z_\alpha$ to zero. The equations
		hold by the definition of the fresh errors, and all bad monomials
		vanish. The map is local. It is surjective because the coefficient
		residues and the native generators give $k_i$, and the
		$b_{i,\lambda}$ contain all of $\mathfrak m_i$.
	\end{proof}
	
	\begin{proposition}\label{prop:localperfect}
		For $D_i=D_{\mathfrak n_iD}$, the quotient
		$D_i/\sigma_iD_i$ is perfect, and there is a surjection
		$D_i\to B_i$ annihilating $\sigma_i$.
	\end{proposition}
	\begin{proof}
		Localize Theorem~\ref{thm:normal} at the selected maximal ideal.
		Lemma~\ref{lem:Pi} says that
		\[
		D_i/(A_i,\sigma_i,\widehat{\calQ}_i)D_i
		\]
		is perfect and maps onto $B_i$.
		
		In $D_i/\sigma_iD_i$, let $E_i$ and $F_i$ be the images of $A_iD_i$
		and $\widehat{\calQ}_iD_i$. Since $H_i\ge v_i(\sigma_i)$, the
		valuation cut gives
		\[
		A_i^2\subseteq\sigma_iA_i\subseteq\sigma_iR_i.
		\]
		Hence $E_i^2=0$. Relation \eqref{eq:Qsquare}, in normalized form,
		gives $F_i^2=0$. By commutativity, $(E_i+F_i)^3=0$.
		Lemma~\ref{lem:lifting} now proves that $D_i/\sigma_iD_i$ is perfect.
		The evaluation map in Lemma~\ref{lem:Pi} gives the stated surjection.
	\end{proof}
	
	\begin{proof}[Proof of Theorem~\ref{thm:main}]
		Construct $D$ as in \eqref{eq:D}, and let
		\[
		g=s_1\cdots s_r.
		\]
		This is a nonzero element of $\Jac(D)$. At the $i$-th maximal
		localization, all foreign $s_j$ are units, and $s_i$ is associated to
		$\sigma_i$. Proposition~\ref{prop:localperfect} therefore shows that
		every maximal localization of $D/gD$ is perfect. The ring is semilocal,
		so Theorem~\ref{thm:bass} gives that $D/gD$ is perfect.
		Proposition~\ref{prop:principal} proves that $D$ is almost perfect.
		
		The evaluations in Proposition~\ref{prop:localperfect} restrict to
		maps from $D_0$ to $B_i$. Every element outside $\mathfrak n_i$ has
		nonzero residue and hence maps to a unit of $B_i$. In particular, all
		denominators used in \eqref{eq:D} map to units in every $B_i$. We
		obtain a homomorphism
		\[
		\theta:D\longrightarrow\prod_{i=1}^rB_i.
		\]
		It sends the coefficient ring $R$ by \eqref{eq:baseeval}, the global
		$x_\lambda$ to $b_\lambda$, and each native $z_\alpha$ to its chosen
		lift in its own coordinate and zero elsewhere. The local normalization
		ratios are congruent to $1$ modulo $A_i$, so these descriptions agree
		with the local evaluations.
		
		Each coordinate map is surjective by Lemma~\ref{lem:Pi}. Moreover,
		$\theta(e_i)$ is the $i$-th decomposition idempotent of $B$, because
		this already holds in the coefficient quotient. Multiplying coordinate
		lifts by these idempotents and adding shows that $\theta$ is surjective
		onto the product, not just onto its separate coordinates.
		
		Let $I=\ker\theta$. At coordinate $i$, $s_i$ is killed and each
		foreign $s_j$ maps to $1$. Hence $g\in I$. Thus $I$ is nonzero and
		proper because $B\ne0$. The first isomorphism theorem gives
		$D/I\cong B$. Finally $R\subseteq D\subseteq E$ has characteristic
		zero, $D$ has dimension one, and $|\Max(D)|=r$ by construction.
	\end{proof}
	
	\begin{corollary}\label{cor:local}
		Every nonzero commutative local perfect ring is a quotient of a
		characteristic-zero local almost perfect domain of dimension one.
	\end{corollary}
	\begin{proof}
		Take $r=1$ in Theorem~\ref{thm:main}.
	\end{proof}
	
	\begin{corollary}\label{cor:equivalence}
		For a nonzero commutative ring $B$, the following are equivalent:
		\begin{enumerate}[(i)]
			\item $B$ is perfect;
			\item $B$ is a quotient of an almost perfect domain by a nonzero
			proper ideal;
			\item $B$ is such a quotient of a characteristic-zero,
			one-dimensional semilocal almost perfect domain.
		\end{enumerate}
	\end{corollary}
	\begin{proof}
		The implication (i)$\Rightarrow$(iii) is the theorem,
		(iii)$\Rightarrow$(ii) is immediate, and (ii)$\Rightarrow$(i) follows
		from the definition of an almost perfect domain.
	\end{proof}
	
	\section{Examples and consequences}
	
	We give examples showing where the different parts of the proof are
	used. None requires the radical to be finitely generated.
	
	\begin{example}\label{ex:unbounded}
		Let $k$ be a field and let
		\[
		M=\bigoplus_{m\ge1} t_m k[t_m]/(t_m^{m+1}),\qquad B=k\oplus M,
		\]
		where products between different summands are zero and the copy
		of $k$ acts by scalars. Then $B$ is local perfect and $M$ is not
		nilpotent.
	\end{example}
	\begin{proof}
		Each element of $M$ has finite support and is nilpotent. Thus $B$ is
		local with maximal ideal $M$. In any sequence in $M$, only the
		finitely many summands occurring in its first term can contribute to
		later initial products. Within the $m$-th such summand, every product
		of $m+1$ radical elements is zero. Taking the maximum over that finite
		support proves $T$-nilpotence. On the other hand, for each $q$ choose
		$m\ge q$; then $t_m^q\ne0$, so $M^q\ne0$. Theorem~\ref{thm:bass}
		applies.
	\end{proof}
	
	Corollary~\ref{cor:local} realizes this ring as a nonzero quotient of
	a local characteristic-zero APD. If $M$ is not nilpotent, the source
	cannot be Noetherian: a nonzero quotient of a one-dimensional
	Noetherian domain is Artinian, and an Artinian local ring has nilpotent
	maximal ideal \cite[Chapter 8]{AM}. Thus non-Noetherian sources are
	necessary for the realization theorem.
	
	\begin{example}
		Let $k$ and $l$ be fields of arbitrary cardinalities, and let $p$ be
		a prime. The perfect ring
		\[
		k\times C/p^nC\times l[\ep]/(\ep^2),
		\]
		where $C$ is any DVR with uniformizer $p$ and $n\ge1$, is a quotient
		of a characteristic-zero APD with three maximal ideals. No
		compatibility between the characteristics of $k,l$ and $p$ is needed.
	\end{example}
	
	This uses the common valued field rather than a common coefficient
	field. A characteristic-zero field factor and a positive-characteristic
	factor cannot be patched by keeping the former as a global subfield of
	the source: that would make all nonzero integers units globally. In
	the construction, the coefficient field occurs only in its own maximal
	localization. The global pullback keeps the appropriate residues at the
	other centers.
	
	\begin{example}
		For $B=\ZZ/p^n\ZZ$ with $n>1$, a parameter mapped to zero cannot be
		chosen equal to $p$. Our construction uses instead
		\[
		p=\tau^{2c+1},\qquad \sigma=\tau^c,\qquad c\ge n.
		\]
		The identity $\sigma^{2c+1}=p^c$ is compatible with the target
		relation $p^n=0$. The normal-form calculation proves more than
		compatibility: all the extra divided products still leave the nonzero
		powers $1,p,\ldots,p^{n-1}$ intact.
	\end{example}
	
	The next two consequences record useful properties of the source
	without further construction.
	
	\begin{corollary}
		The source $D$ in Theorem~\ref{thm:main} can be chosen with
		$\Spec(D)=\{(0),\mathfrak n_1,\ldots,\mathfrak n_r\}$ and with
		\[
		\Frac(D)=D[1/g]=\bigcup_{m\ge0}g^{-m}D.
		\]
	\end{corollary}
	\begin{proof}
		Its dimension and semilocality give the stated spectrum. For
		$0\ne b\in D$, the proof of Proposition~\ref{prop:principal} gives
		$g^m\in bD$ for some $m$, hence $1/b\in D[1/g]$.
	\end{proof}
	
	\begin{corollary}
		If $J$ is any ideal of a commutative perfect ring $B$, then the same
		domain $D$ realizing $B$ also realizes $B/J$ whenever $B/J\ne0$, with
		a nonzero kernel.
	\end{corollary}
	\begin{proof}
		Compose $D\twoheadrightarrow B$ with $B\twoheadrightarrow B/J$. The
		new kernel contains the original nonzero element $g$.
	\end{proof}
	
	The proof separates three kinds of relations. The coefficient
	quotients keep the finite characteristic. Triangular monic equations
	keep the residue extension. Bad monomials keep the sequential
	vanishing needed for perfectness. Additional additive relations in $B$
	are imposed only by the final evaluation map. This separation explains
	why no nilpotence bound on the whole radical is necessary, and why
	purely inseparable residue extensions need not split inside the target
	ring.

\end{document}